\documentclass[a4paper,twoside,10pt]{article}

\usepackage[a4paper,width=150mm,top=25mm,bottom=25mm,bindingoffset=6mm]{geometry}

\usepackage[english]{babel}

\usepackage{graphicx}
\usepackage[font=small,labelfont=bf]{caption}
\usepackage{subcaption}
\usepackage{color}
\usepackage{mdframed}
\usepackage{framed} 
\usepackage{cancel}
\usepackage[normalem]{ulem}

\usepackage{lineno,hyperref}
\modulolinenumbers[5]

\usepackage{amsmath,amssymb,amsthm,mathtools,mathrsfs,epsfig,amsfonts,wasysym}
\usepackage{MnSymbol}

\usepackage[normalem]{ulem}
\usepackage{ragged2e}
\usepackage{csquotes}
\usepackage{tasks}
\usepackage{paralist}
\usepackage[inline]{enumitem}
\usepackage{tikz-cd}
\usepackage{stmaryrd}

\usepackage[titletoc,title]{appendix}

\newcommand{\e}{\mathrm{e}}

\usepackage{url,hyperref}
\hypersetup{colorlinks=true,linkcolor=blue,filecolor=magenta,urlcolor=cyan}
\usepackage{nicematrix}
\usepackage{tikz}
\usetikzlibrary{fit}

\usepackage[all]{xy}

\newtheorem{theorem}{Theorem}
\newtheorem{proposition}{Proposition}

\usepackage{authblk}
\usepackage{todonotes}

\definecolor{wildstrawberry}{rgb}{1.0, 0.26, 0.64}

\definecolor{ao(english)}{rgb}{0.0, 0.5, 0.0}

\usepackage{orcidlink}

\begin{document}

\title{Generalized fraction rules for monotonicity with higher antiderivatives and derivatives}

\author[1]{Vasiliki Bitsouni\,\orcidlink{0000-0002-0684-0583}\thanks{\texttt{vbitsouni@math.uoa.gr}}}

\author[1,2]{Nikolaos Gialelis\,\orcidlink{0000-0002-6465-7242}\thanks{\texttt{ngialelis@math.uoa.gr}}}

\author[3]{Dan \c{S}tefan Marinescu\,\orcidlink{0000-0001-9066-2558}\thanks{\texttt{marinescuds@gmail.com}}}

\affil[1]{Department of Mathematics, National and Kapodistrian University of Athens, Panepistimioupolis, GR-15784 Athens, Greece}

\affil[2]{School of Medicine, National and Kapodistrian University of Athens,\newline GR-11527 Athens, Greece}

\affil[3]{National College \textquote{Iancu de Hunedoara}, Hunedoara, Romania}

\date{}

\maketitle

\begin{abstract}
\noindent
We first introduce the generic versions of the fraction rules for monotonicity, i.e. the one that involves integrals known as the Gromov theorem and the other that involves derivatives known as L'H\^opital rule for monotonicity, which we then extend to high order antiderivatives and derivatives, respectively. 
\end{abstract}

\noindent
\textbf{Keywords:} fraction rules for monotonicity, the Gromov theorem, the L'H\^opital rule for monotonicity, high order antiderivative, high order mean, Cauchy formula of repeated integration, high order derivative, Taylor polynomial, Taylor remainder 

\noindent
\textbf{MSC2020-Mathematics Subject Classification System:} 26A24, 26A36, 26A48, 26D10



\section{Introduction} 

Roughly speaking, the application of either the integral or the differential operation to both the numerator and the denominator of a fraction, preserves the monotonicity of the fraction. The integral case of such fact is known as \textit{the Gromov theorem} (see, e.g., \cite{chavel2006riemannian, estrada2017hopital}), while the differential case is called \textit{the L'H\^opital rule for monotonicity} (see, e.g., \cite{pinelis2002hospital, pinelis2001hospital, anderson2006monotonicity, wu2009generalization, estrada2017hopital}). The Gromov theorem first appeared in \cite{cheeger1982finite}, i.e. about a decade \textit{before} the introduction of the L'H\^opital rule for monotonicity in \cite{anderson1993inequalities}. 

These results have been proven to be quite useful analytical tools with many applications to a plethora of mathematical areas, such as differential geometry (see, e.g., \cite{cheeger1982finite, chavel2006riemannian}), quasiconformal theory (see, e.g., \cite{anderson1993inequalities}), information theory (see, e.g., \cite{pinelis2002hospital}), probability theory (see, e.g., \cite{pinelis2002hospitalpr}), approximation theory (see, e.g., \cite{pinelis2002monotonicity}), theory of special functions (see, e.g., \cite{anderson2006monotonicity, wu2009generalization, chen2022monotonicity}) and theory of analytic functions (see, e.g., \cite{estrada2017hopital}). 

Below follow the most generic versions of these fraction rules for monotonicity, for the statement of which we remind that a real function, defined in an interval of the extended real line, $\left[-\infty,\infty\right]$, is locally characterized by a property when it is characterized by that property in every compact subinterval of its domain (we remind that an unbounded interval of the form $\left[-\infty,\infty\right]$, $\left[-\infty,a\right]$ or $\left[a,\infty\right]$, for some $a\in\mathbb{R}$, is compact).

\begin{theorem}[the Gromov theorem]
\label{r0}
Consider
\begin{enumerate}
\item an interval $I\subseteq\left[-\infty,\infty\right]$, 
\item a point $c\in I$ and 
\item two functions $f,g\,\colon\,I\to\mathbb{R}$, such that 
\begin{enumerate}[label=\roman*.]
\item $f$ and $g$ are both locally Lebesgue integrable and 
\item $g$ preserves Lebesgue-almost everywhere a non zero sign.
\end{enumerate}
\end{enumerate}
If $\frac{f}{g}\colon\,I\to\left[-\infty,\infty\right]$ is Lebesgue-almost everywhere (strictly) monotonic, then $\frac{\int\limits_{c}^\cdot{f{\left(t\right)}\mathrm{d}t}}{\int\limits_{c}^\cdot{g{\left(t\right)}\mathrm{d}t}}\colon\,I\setminus\left\{c\right\}\to\mathbb{R}$ is (strictly) monotonic of the same (strict) monotonicity. 
\end{theorem} 

\begin{theorem}[the L'H\^opital rule for monotonicity]
\label{r1}
Consider 
\begin{enumerate}
\item an interval $I\subseteq\left[-\infty,\infty\right]$, 
\item a point $c\in I$ and 
\item two functions $f,g\,\colon\,I\to\mathbb{R}$, such that 
\begin{enumerate}[label=\roman*.]
\item $\left.f\right|_{I\cap\mathbb{R}}$ and $\left.g\right|_{I\cap\mathbb{R}}$ are both differentiable and 
\item $g'{\left(x\right)}\neq 0$, for all $x\in I\cap\mathbb{R}$.
\end{enumerate}
\end{enumerate}
If $\frac{f'}{g'}\colon\,I\cap\mathbb{R}\to\mathbb{R}$ is (strictly) monotonic, then $\frac{f-f{\left(c\right)}}{g-g{\left(c\right)}}\colon\,I\setminus\left\{c\right\}\to\mathbb{R}$ is (strictly) monotonic of the same (strict) monotonicity. 
\end{theorem} 

It can be shown (see \hyperref[equivalence]{\S \ref*{equivalence}}) that \hyperref[r0]{Theorem \ref*{r0}} is stronger than \hyperref[r1]{Theorem \ref*{r1}}, a fact that has already been observed in \cite{estrada2017hopital}. However, the latter one is an independent result of differential calculus, for the proof of which no tools of the integration theory are needed (see \hyperref[the lhopital rule for monotonicity via differential calculus]{\S \ref*{the lhopital rule for monotonicity via differential calculus}}). 

The goal of the present manuscript is not only the proof of \hyperref[r0]{Theorem \ref*{r0}} and \hyperref[r1]{Theorem \ref*{r1}}, but also the introduction of their generalizations to higher antiderivatives and derivatives, respectively. Our analysis is organized as follows. In \hyperref[basic notions]{\S \ref*{basic notions}} we review some necessary notions used for the compact statement of the aforementioned generalizations. In \hyperref[generalized fraction rules for monotonicity]{\S \ref*{generalized fraction rules for monotonicity}}, after the statement of the main results, we examine the relation between them and we proceed to their proof. In \hyperref[corollaries and examples]{\S \ref*{corollaries and examples}} we employ our findings in some novel applications. In \hyperref[the lhopital rule for monotonicity via differential calculus]{\S \ref*{the lhopital rule for monotonicity via differential calculus}} we provide an alternative proof of the generalized L'H\^opital rule for monotonicity with the exclusive utilization of the differential calculus toolbox. 

\section{Basic notions}
\label{basic notions}  

For the statement of our results we make a short, necessary note on the notation used. 
\begin{enumerate}
\item For every 
\begin{enumerate}[label=\roman*.]
\item $n\in\mathbb{N}$, 
\item interval $I\subseteq\left[-\infty,\infty\right]$ when $n=1$ or $I\subseteq\mathbb{R}$ when $n\neq 1$, 
\item $c\in I$ and
\item locally Lebesgue integrable $f\colon I\to\mathbb{R}$, 
\end{enumerate}
$A_{n,f,c}$ stands for \textit{the antiderivative} of order $n$ for $f$ at $c$, i.e. 
\begin{align*}
A_{n,f,c}\colon\,I&\to\mathbb{R}\\
x&\mapsto A_{n,f,c}{\left(x\right)}\coloneqq \frac{1}{\left(n-1\right)!}\int\limits_{c}^{x}{f{\left(t\right)}{\left(x-t\right)}^{n-1}\mathrm{d}t}.
\end{align*} 
The name of this function is nothing but random. It comes from \textit{the Cauchy formula of repeated integration}, $$A_{n,f,c}=\int\limits_{c}^{\cdot}{{\int\limits_{c}^{t_1}{{\cdots{\int\limits_{c}^{t_{n-1}}{f{\left(t_n\right)}\mathrm{d}t_n}}\dots}\mathrm{d}t_2}}\mathrm{d}t_1},\text{ when }n\neq 1.$$ This formula is introduced in \cite[Trente-Cinqui\`{e}me Le\c{c}on in page 137]{cauchy1823re} with the additional assumption of $f$ being continuous. The result is then derived from the density of continuous functions in the space of integrable ones (see, e.g., \cite[Theorem 11.5.8 in page 391]{choudary2014real}). With this equality at hand we can directly verify that 
\begin{equation}
\label{knk}
A_{n,f,c}=A_{k,A_{n-k,f,c},c},\text{ }\forall k\left\{1,\dots,n-1\right\},\text{ when }n\neq 1.
\end{equation}
Moreover, by the use of the fundamental theorem of calculus (see, e.g., \cite[Theorem B.4.3 in page 497]{choudary2014real}) we obtain that, when $n\neq 1$, the function $\left.A_{n,f,c}\right|_{I\cap\mathbb{R}}$ is $\left(n-1\right)$-times differentiable and $n$-times Lebesgue-almost everywhere differentiable, with $${A_{n,f,c}}^{\left(k\right)}=A_{n-k,f,c},\text{ }\forall k\in\left\{1,\dots,n-1\right\}$$ and $${A_{n,f,c}}^{\left(n\right)}=f,\text{ Lebesgue-almost everywhere}.$$ If, in addition, $f$ is continuous, then $\left.A_{n,f,c}\right|_{I\cap\mathbb{R}}$ is $n$-times differentiable, with 
\begin{equation}
\label{cAnnf}
{A_{n,f,c}}^{\left(n\right)}=f.
\end{equation}
\item For every 
\begin{enumerate}[label=\roman*.] 
\item $n\in\mathbb{N}$, 
\item interval $I\subseteq\mathbb{R}$, 
\item $c\in I$ and 
\item locally Lebesgue integrable $f\colon,I\to\mathbb{R}$, 
\end{enumerate}
$M_{n,f,c}$ stands for \textit{the mean} of order $n$ for $f$ at $c$, i.e. 
\begin{align*}
M_{n,f,c}\colon\,I\setminus\left\{c\right\}&\to\mathbb{R}\\
x&\mapsto M_{n,f,c}{\left(x\right)}\coloneqq \frac{n}{{\left(x-c\right)}^n}\int\limits_{c}^{x}{f{\left(t\right)}{\left(x-t\right)}^{n-1}\mathrm{d}t}.
\end{align*}
The concept behind the above definition lies in the observation that $$A_{n,1,c}{\left(x\right)}=\frac{{\left(x-c\right)}^n}{n!},\text{ }\forall x\in \mathbb{R},$$ which confirms the expected equality $$M_{n,f,c}=\frac{A_{n,f,c}}{A_{n,1,c}}.$$
\item For every 
\begin{enumerate}[label=\roman*.]
\item $n\in\mathbb{N}_0$, 
\item interval $I\subseteq\left[-\infty,\infty\right]$ when $n=0$ or $I\subseteq\mathbb{R}$ when $n\neq 0$, 
\item $c\in I$ and 
\item $n$-times differentiable in $I\cap\mathbb{R}$ $f\colon I\to\mathbb{R}$, 
\end{enumerate}
$T_{n,f,c}$ and $R_{n,f,c}$ stand for \textit{the Taylor polynomial and remainder}, respectively, of order $n$ for $f$ at $c$, i.e. 
\begin{align*}
T_{n,f,c}\colon\,I&\to\mathbb{R}\\
x&\mapsto T_{n,f,c}{\left(x\right)}\coloneqq \sum\limits_{k=0}^n{\frac{f^{\left(k\right)}{\left(c\right)}}{k!}{\left(x-c\right)}^k},
\end{align*}
and 
\begin{align*}
R_{n,f,c}\colon\,I&\to\mathbb{R}\\
x&\mapsto R_{n,f,c}{\left(x\right)}\coloneqq f{\left(x\right)}-T_{n,f,c}{\left(x\right)}.
\end{align*}
If, in addition, $n\in\mathbb{N}$ and $f^{\left(n\right)}\colon\,I\cap\mathbb{R}\to\mathbb{R}$ is locally Lebesgue integrable, then \textit{the integral form of the remainder} (see, e.g., \cite[\S{1.6} in page 62]{bourbaki2004elements}) implies that 
\begin{equation}
\label{Rn1fAnfn}
R_{n-1,f,c}=A_{n,f^{\left(n\right)},c}.
\end{equation}
\end{enumerate}

\section{Generalized fraction rules for monotonicity}
\label{generalized fraction rules for monotonicity} 

\subsection{Statement}
\label{statement} 

For the proper statement of the main results, we need the following result. 

\begin{proposition}
\label{prqst}
Consider
\begin{enumerate}[label=\roman*.]
\item a natural number $n\in\mathbb{N}$, 
\item an interval $I\subseteq\left[-\infty,\infty\right]$ when $n=1$ or $I\subseteq\mathbb{R}$ when $n\neq 1$, 
\item a point $c\in I$ and 
\item a function $f\colon\,I\to\mathbb{R}$. 
\end{enumerate}
\begin{enumerate}
\item If $f$  
\begin{enumerate}[label=\alph*.]
\item is locally Lebesgue integrable and 
\item preserves Lebesgue-almost everywhere a non zero sign,
\end{enumerate}
then ${A_{n,f,c}}^{-1}{\left(\left\{0\right\}\right)}=\left\{c\right\}$. 
\item If  
\begin{enumerate}[label=\alph*.]
\item $\left.f\right|_{I\cap\mathbb{R}}$ is $n$-times differentiable and
\item $f^{\left(n\right)}{\left(x\right)}\neq 0$, for all $x\in I\cap\mathbb{R}$, 
\end{enumerate}
then ${R_{n-1,f,c}}^{-1}{\left(\left\{0\right\}\right)}=\left\{c\right\}$. 
\end{enumerate}
\end{proposition}
\begin{proof}
\begin{enumerate}
\item To begin with, we have that $A_{n,f,c}{\left(c\right)}=0$.  

Since $f$ preserves Lebesgue-almost everywhere a non zero sign, we deduce that for every $x\in I\setminus\left\{c\right\}$ the function $\left(x-\mathrm{id}\right)^{n-1}f\colon\,\left(\min{\left\{c,x\right\}},\max{\left\{c,x\right\}}\right)\to\mathbb{R}$ also preserves Lebesgue-almost everywhere a non zero sign, where $\mathrm{id}$ stands for the identity function. Thus $A_{n,f,c}{\left(x\right)}\neq 0$ and the result then follows. 
\item We have that $R_{n-1,f,c}{\left(c\right)}=0$. 

Since $f^{\left(n\right)}{\left(x\right)}\neq 0$, for all $x\in I\cap\mathbb{R}$, from the Darboux theorem (see, e.g., \cite[Theorem 8.3.2 in page 228]{choudary2014real}) we have that $f^{\left(n\right)}$ preserves a non zero sign, that is $f^{\left(n-1\right)}\colon\,I\cap\mathbb{R}\to\mathbb{R}$ is strictly monotonic, hence $f^{\left(n\right)}$ is locally Lebesgue integrable (see, e.g., \cite[Theorem B.2.5 in 490]{choudary2014real}). 

Now, we first apply point 1. for the function $f^{\left(n\right)}$ and we then employ \eqref{Rn1fAnfn}, in order to get the desired result. 
\end{enumerate}
\end{proof}

With \hyperref[prqst]{Proposition \ref*{prqst}} at hand, we can now state the generalizations of \hyperref[r0]{Theorem \ref*{r0}} and \hyperref[r1]{Theorem \ref*{r1}} to higher antiderivatives and derivatives, respectively. 

\begin{theorem}[generalization to higher antiderivatives]
\label{r00}
Consider
\begin{enumerate}
\item a natural number $n\in\mathbb{N}$, 
\item an interval $I\subseteq\left[-\infty,\infty\right]$ when $n=1$ or $I\subseteq\mathbb{R}$ when $n\neq 1$, 
\item a point $c\in I$ and 
\item two functions $f,g\,\colon\,I\to\mathbb{R}$, such that 
\begin{enumerate}[label=\roman*.]
\item $f$ and $g$ are both locally Lebesgue integrable and 
\item $g$ preserves Lebesgue-almost everywhere a non zero sign.
\end{enumerate}
\end{enumerate}
If $\frac{f}{g}\colon\,I\to\left[-\infty,\infty\right]$ is Lebesgue-almost everywhere (strictly) monotonic, then $\frac{A_{n,f,c}}{A_{n,g,c}}\colon\,I\setminus\left\{c\right\}\to\mathbb{R}$ is (strictly) monotonic of the same (strict) monotonicity. 
\end{theorem}

\begin{theorem}[generalization to higher derivatives]
\label{r2}
Consider
\begin{enumerate}
\item a natural number $n\in\mathbb{N}$, 
\item an interval $I\subseteq\left[-\infty,\infty\right]$ when $n=1$ or $I\subseteq\mathbb{R}$ when $n\neq 1$, 
\item a point $c\in I$ and 
\item two functions $f,g\,\colon\,I\to\mathbb{R}$, such that 
\begin{enumerate}[label=\roman*.]
\item $\left.f\right|_{I\cap\mathbb{R}}$ and $\left.g\right|_{I\cap\mathbb{R}}$ are both $n$-times differentiable and
\item $g^{\left(n\right)}{\left(x\right)}\neq 0$, for all $x\in I\cap\mathbb{R}$. 
\end{enumerate}
\end{enumerate}
If $\frac{f^{\left(n\right)}}{g^{\left(n\right)}}\colon\,I\cap\mathbb{R}\to\mathbb{R}$ is (strictly) monotonic, then $\frac{R_{n-1,f,c}}{R_{n-1,g,c}}\colon\,I\setminus\left\{c\right\}\to\mathbb{R}$ is (strictly) monotonic of the same (strict) monotonicity. 
\end{theorem} 

\subsection{Equivalence?}
\label{equivalence} 

In general, \hyperref[r00]{Theorem \ref*{r00}} is stronger than \hyperref[r2]{Theorem \ref*{r2}}. 

\begin{proposition}
\label{str1}
\hyperref[r00]{Theorem \ref*{r00}} implies \hyperref[r2]{Theorem \ref*{r2}}. 
\end{proposition}
\begin{proof}
Under the hypothesis of \hyperref[r2]{Theorem \ref*{r2}}, we first deduce that both $f^{\left(n\right)},g^{\left(n\right)}\colon\, I\cap\mathbb{R}\to\mathbb{R}$ are locally Lebesgue integrable. Indeed, we can argue as in the proof of point 2. of \hyperref[prqst]{Proposition \ref*{prqst}}, in order to show that $g^{\left(n\right)}$ is locally Lebesgue integrable. Moreover, $\frac{f^{\left(n\right)}}{g^{\left(n\right)}}$ is locally bounded since it is (strictly) monotonic, hence we write $$f^{\left(n\right)}=\frac{f^{\left(n\right)}}{g^{\left(n\right)}}g^{\left(n\right)}$$ and we conclude that $f^{\left(n\right)}$ is also locally Lebesgue integrable as a product of a locally bounded function and a locally Lebesgue integrable one. 

Now, we first apply \hyperref[r00]{Theorem \ref*{r00}} for the functions $f^{\left(n\right)}$ and $g^{\left(n\right)}$ and we then employ \eqref{Rn1fAnfn}. 
\end{proof}

We can weaken \hyperref[r00]{Theorem \ref*{r00}} in a specific manner, in order to get the reverse implication of \hyperref[str1]{Proposition \ref*{str1}}. 

\begin{proposition}
\label{str2}
\hyperref[r2]{Theorem \ref*{r2}} implies \hyperref[r00]{Theorem \ref*{r00}}, when the latter one is equipped with the hypothesis that $f$ and $g$ are both continuous instead of being just locally Lebesgue integrable.  
\end{proposition}
\begin{proof}
Under the hypothesis of the weakened \hyperref[r00]{Theorem \ref*{r00}}, \eqref{cAnnf} implies that $\left.A_{n,f,c}\right|_{I\cap\mathbb{R}}$ and $\left.A_{n,g,c}\right|_{I\cap\mathbb{R}}$ are both $n$-times differentiable.  

Now, all we have to do is to apply \hyperref[r2]{Theorem \ref*{r2}} for the functions $A_{n,f,c}$ and $A_{n,g,c}$. 
\end{proof}

\subsection{Proof}
\label{proof} 

In view of \hyperref[str1]{Proposition \ref*{str1}}, we only need to prove the stronger of the main results, in particular, \hyperref[r00]{Theorem \ref*{r00}}.

\begin{proof}[Proof of Theorem 5.] 
It suffices to show the result only for the case where $g$ preserves Lebesgue-almost everywhere the positive sign. Indeed, we can employ such a result for $-f$ and $-g$ instead of $f$ and $g$, respectively, in order to get the corresponding one for $g$ that preserves Lebesgue-almost everywhere the negative sign. 

Moreover, it suffices to show \hyperref[r00]{Theorem \ref*{r00}} only for the case where $\frac{f}{g}$ is Lebesgue-almost everywhere (strictly) increasing. Indeed, we can employ such a result for $-f$ instead of $f$, in order to get the corresponding one for $\frac{f}{g}$ that is Lebesgue-almost everywhere (strictly) decreasing.

Hence, we assume, without loss of generality, that $g$ preserves Lebesgue-almost everywhere the positive sign and that $f$ is Lebesgue-almost everywhere (strictly) increasing.  

We will show the desired result by induction on $n$. 

\begin{enumerate}
\item \textit{The base case}. The case where $n=1$ is nothing but \hyperref[r0]{Theorem \ref*{r0}} itself. 

Since $g$ preserves Lebesgue-almost everywhere the positive sign, the function $A_{1,g,c}$ is strictly increasing, which implies that its inverse ${A_{1,g,c}}^{-1}\colon\,A_{1,g,c}{\left(I\right)}\to I$ is not only well defined but also strictly increasing. In addition, the continuity of $A_{1,g,c}$ guarantees that $A_{1,g,c}{\left(I\right)}$ is an interval. 

We then consider the function $h\coloneqq A_{1,f,c}\circ {A_{1,g,c}}^{-1}\colon\,A_{1,g,c}{\left(I\right)}\to \mathbb{R}$ and we claim that $$h=A_{1,\frac{f}{g}\circ{A_{1,g,c}}^{-1},0},$$ that is $$\int\limits_{c}^{{A_{1,g,c}}^{-1}{\left(\cdot\right)}}{f{\left(t\right)}\mathrm{d}t}=\int\limits_{0}^{\cdot}{\frac{f{\left({A_{1,g,c}}^{-1}{\left(t\right)}\right)}}{g{\left({A_{1,g,c}}^{-1}{\left(t\right)}\right)}}\mathrm{d}t}.$$ Indeed, observing that $$\int\limits_{c}^{{A_{1,g,c}}^{-1}{\left(\cdot\right)}}{f{\left(t\right)}\mathrm{d}t}=\int\limits_{c}^{{A_{1,g,c}}^{-1}{\left(\cdot\right)}}{\frac{f{\left(t\right)}}{g{\left(t\right)}}g{\left(t\right)}\mathrm{d}t}=\int\limits_{c}^{{A_{1,g,c}}^{-1}{\left(\cdot\right)}}{\frac{f{\left(t\right)}}{g{\left(t\right)}}{A_{1,g,c}}'{\left(t\right)}\mathrm{d}t},$$ we get the desired equality by the use of the change of variable formula (see, e.g., \cite[point (i) of Corollary 6.97 in page 326]{stromberg2015introduction}). 

Moreover, $\frac{f}{g}\circ{A_{1,g,c}}^{-1}\colon\,A_{1,g,c}{\left(I\right)}\to\left[-\infty,\infty\right]$ is Lebesgue-almost everywhere (strictly) increasing as a composition of a strictly increasing function and an Lebesgue-almost everywhere (strictly) increasing function.

The combination of the above two facts implies that $h$ is (strictly) convex (see, e.g., \cite[Theorem A in page 9 and Remark B in page 13]{roberts1973convex} or \cite[Theorem 14.14 in page 334]{yeh2014real}). Hence, from the equality $h{\left(0\right)}=0$ along with the Galvani lemma (see, e.g., \cite[Theorem 1.3.1 in page 20]{niculescu2006convex}) we deduce that the function $$\frac{h}{\mathrm{id}}\colon\,A_{1,g,c}{\left(I\right)}\setminus\left\{0\right\}\to\mathbb{R}$$ is (strictly) increasing and so is $$\frac{h\circ A_{1,g,c}}{A_{1,g,c}}\colon\,I\setminus\left\{c\right\}\to\mathbb{R},$$ since $A_{1,g,c}$ is strictly increasing. The result then follows from the fact that $h\circ A_{1,g,c}=A_{1,f,c}$. 
\item \textit{The induction step}. If $n\neq 1$, we then fix a natural number $k\in\left\{1,\dots,n-1\right\}$. In view of point 1. of \hyperref[prqst]{Proposition \ref*{prqst}}, both $\frac{A_{k,f,c}}{A_{k,g,c}},\frac{A_{k+1,f,c}}{A_{k+1,g,c}}\,\colon I\setminus\left\{c\right\}\to\mathbb{R}$ are well defined. 

We assume that $\frac{A_{k,f,c}}{A_{k,g,c}}$ is (strictly) increasing and we will show that $\frac{A_{k+1,f,c}}{A_{k+1,g,c}}$ is (strictly) increasing. 

We consider the functions $$\tilde{f}\coloneqq{\left(\mathrm{sgn}\circ\left(\mathrm{id}-c\right)\right)}^k A_{k,f,c}\colon\,I\to\mathbb{R}\text{ and }\tilde{g}\coloneqq{\left(\mathrm{sgn}\circ\left(\mathrm{id}-c\right)\right)}^k A_{k,g,c}\colon\,I\to\mathbb{R},$$ which are both locally Lebesgue integrable. 

We claim that $\tilde{g}$ preserves Lebesgue-almost everywhere the positive sign. Indeed, we have that $$\tilde{g}{\left(x\right)}=\frac{\mathrm{sgn}{\left(x-c\right)}}{\left(k-1\right)!}\int_{c}^{x}{g{\left(t\right)}{\left|x-t\right|}^{k-1}\mathrm{d}t},\text{ }\forall x\in I,$$ since $$\mathrm{sgn}{\left(x-c\right)}=\mathrm{sgn}{\left(x-t\right)},\text{ }\forall t\in \left(\min{\left\{c,x\right\}},\max{\left\{c,x\right\}}\right),\text{ }\forall x\in I\setminus\left\{c\right\},$$ therefore $\left.\tilde{g}\right|_{I\setminus\left\{c\right\}}$ preserves the positive sign. 

In addition, $\frac{\tilde{f}}{\tilde{g}}\colon\,I\setminus\left\{c\right\}\to\mathbb{R}$ is (strictly) increasing, since $$\frac{\tilde{f}}{\tilde{g}}=\frac{A_{k,f,c}}{A_{k,g,c}}.$$ 

With the above facts at hand, all we have to do is first to apply \hyperref[r0]{Theorem \ref*{r0}} for the functions $\tilde{f}$ and $\tilde{g}$ and second to employ \eqref{knk}, in order to obtain the desired result. 
\end{enumerate} 
\end{proof}

\section{Corollaries and examples}
\label{corollaries and examples} 

Below follow some applications of the generalized fraction rules for monotonicity. 

\begin{enumerate}
\item \textit{Monotonicity of high order mean}: We consider
\begin{enumerate}[label=\roman*.]
\item a natural number $n\in\mathbb{N}$, 
\item an interval $I\subseteq\mathbb{R}$, 
\item a point $c\in I$ and 
\item a locally Lebesgue integrable function $f\colon\,I\to\mathbb{R}$. 
\end{enumerate}
If $f$ is Lebesgue-almost everywhere (strictly) monotonic, then from \hyperref[r00]{Theorem \ref*{r00}} for $g\equiv 1$ we deduce that $M_{n,f,c}$ is (strictly) monotonic of the same (strict) monotonicity. 
\item \textit{Convexity of high order mean}: We consider
\begin{enumerate}[label=\roman*.]
\item a natural number $n\in\mathbb{N}$, 
\item an interval $I\subseteq\mathbb{R}$, 
\item a point $c\in I$ and 
\item a convex function $f\colon\,I\to\mathbb{R}$. 
\end{enumerate}
From the Galvani lemma we have that the function $\frac{f-f{\left(c\right)}}{\mathrm{id}-c}\colon\,I\setminus\left\{c\right\}\to\mathbb{R}$ is (strictly) increasing. Extending the above function as $$\frac{\left(f-f{\left(c\right)}\right)\mathrm{sgn}\circ\left(\mathrm{id}-c\right)}{\left|\mathrm{id}-c\right|}\colon\,I\to\left[-\infty,\infty\right]$$ and remembering that every convex function is locally Lebesgue integrable, we employ \hyperref[r00]{Theorem \ref*{r00}}, in order to obtain that the function $$\frac{A_{n,\left(f-f{\left(c\right)}\right)\mathrm{sgn}\circ\left(\mathrm{id}-c\right),c}}{\left(n+1\right)A_{n,\left|\mathrm{id}-c\right|,c}}=\frac{A_{n,f-f{\left(c\right)},c}}{\left(n+1\right)A_{n,\mathrm{id}-c,c}}=\frac{M_{n,f-f{\left(c\right)},c}}{\left(n+1\right)M_{n,\mathrm{id}-c,c}}=\frac{M_{n,f,c}-f{\left(c\right)}}{\mathrm{id}-c}\colon\,I\setminus\left\{c\right\}\to\mathbb{R}$$ is also (strictly) increasing. Hence, again from the Galvani lemma we deduce that $M_{n,f,c}$ is (strictly) convex.  
\item \textit{An application in ordinary differential equations}: We consider the classic nondimensionalized epidemiological model of the single epidemic outbreak for non negative times $t\in\left[0,\infty\right)$, 
\begin{align*}
S'{\left(t\right)}&=-\mathcal{R}_0 S{\left(t\right)} I{\left(t\right)}\\
I'{\left(t\right)}&=-I{\left(t\right)}+\mathcal{R}_0 S{\left(t\right)} I{\left(t\right)}\\
R'{\left(t\right)}&=I{\left(t\right)}, 
\end{align*}
where $\mathcal{R}_0>1$ and we search for $S,I,R\,\colon\left[0,\infty\right]\to\left[0,1\right]$, when the initial values $S{\left(0\right)}$, $I{\left(0\right)}$ and $R{\left(0\right)}$ are given. 

In the non trivial epidemiological situation of $S{\left(0\right)},I{\left(0\right)}\,\in\left(0,1\right)$ and $R{\left(0\right)}\in\left[0,1\right)$, there exists such functions satisfying the following properties,
\begin{enumerate}[label=\roman*.]
\item $S{\left(t\right)},I{\left(t\right)}\,\in\left(0,1\right)$, for every $t\in\left[0,\infty\right)$, with $$I{\left(t\right)}+S{\left(t\right)}-\frac{1}{\mathcal{R}_0}\ln{S{\left(t\right)}}=I{\left(0\right)}+S{\left(0\right)}-\frac{1}{\mathcal{R}_0}\ln{S{\left(0\right)}},\text{ }\forall t\in\left[0,\infty\right)$$ and 
\item $I{\left(\infty\right)}=0$ and $$S{\left(\infty\right)}=-\frac{1}{\mathcal{R}_0}W{\left(-\mathcal{R}_0S{\left(0\right)}\e^{-\mathcal{R}_0\left(S{\left(0\right)}+I{\left(0\right)}\right)}\right)}\in\left(0,\frac{1}{\mathcal{R}_0}\right),$$ where $W$ stands for the Lambert function (see, e.g., \cite{corless1996lambert}). 
\end{enumerate}

Hence, $S'{\left(t\right)}<0$, for all $t\in\left[0,\infty\right)$, which implies that $S$ is strictly decreasing. By the use of \hyperref[r1]{Theorem \ref*{r1}} we deduce that $\frac{I-I{\left(c\right)}}{S-S{\left(c\right)}}\colon\left[0,\infty\right]\setminus\left\{c\right\}\to\mathbb{R}$ is strictly increasing for every $c\in\left[0,\infty\right]$, since $$\frac{I'}{S'}=\frac{1}{\mathcal{R}_0 S}-1$$ is strictly increasing. Thus, $$\lim\limits_{t\to c}{\frac{I{\left(t\right)}-I{\left(c\right)}}{S{\left(t\right)}-S{\left(c\right)}}}\overset{\frac{0}{0}}{=}\lim\limits_{t\to c}{\frac{I'{\left(t\right)}}{S'{\left(t\right)}}}=\frac{1}{\mathcal{R}_0 S{\left(c\right)}}-1.$$ These facts imply that $$I{\left(t\right)}<I{\left(c\right)}+\left(\frac{1}{\mathcal{R}_0S{\left(c\right)}}-1\right)\left(S{\left(t\right)}-S{\left(c\right)}\right),\text{ }\forall t\in\left[0,\infty\right]\setminus\left\{c\right\},$$ i.e. a useful \textit{a priori} estimate when $c=0$.  

Moreover, by the use of \hyperref[r00]{Theorem \ref*{r00}} we deduce that $\frac{M_{n,I,c}-I{\left(c\right)}}{M_{n,S,c}-S{\left(c\right)}}\colon\left[0,\infty\right)\setminus\left\{c\right\}\to\mathbb{R}$ is strictly increasing for every $n\in\mathbb{N}$ and $c\in\left[0,\infty\right)$. The corresponding inequality is $$M_{n,I,c}{\left(t\right)}<I{\left(c\right)}+\left(\frac{1}{\mathcal{R}_0S{\left(c\right)}}-1\right)\left(M_{n,S,c}{\left(t\right)}-S{\left(c\right)}\right),\text{ }\forall t\in\left[0,\infty\right)\setminus\left\{c\right\}.$$ which can also be deduced directly from the previous one. 
\item \textit{Multidimensional analogue for specific radial functions}: We consider 
\begin{enumerate}[label=\roman*.]
\item a natural number $n\in\mathbb{N}$ and 
\item two functions $f,g\,\colon\,\left[0,\infty\right)\to\mathbb{R}$, such that 
\begin{enumerate}[label=\alph*.]
\item $f$ and $g$ are both locally Lebesgue integrable and 
\item $g$ preserves Lebesgue-almost everywhere a non zero sign.
\end{enumerate}
\end{enumerate}
We then set $$\begin{aligned}
\phi\colon\,\coprod\limits_{r\in\left(0,\infty\right)}{B{\left(0_n,r\right)}}&\to\mathbb{R}\\
\left(r,x\right)&\mapsto \phi{\left(r,x\right)}\coloneqq f{\left(r-\left|x\right|\right)}
\end{aligned}\text{ and }\begin{aligned}
\psi\colon\,\coprod\limits_{r\in\left(0,\infty\right)}{B{\left(0_n,r\right)}}&\to\mathbb{R}\\
\left(r,x\right)&\mapsto \psi{\left(r,x\right)}\coloneqq g{\left(r-\left|x\right|\right)},
\end{aligned}$$ where $B{\left(0_n,r\right)}$ stands for the $n$-dimensional ball of radius $r>0$ centered at the origin $0_n\in\mathbb{R}^n$ and $\left|\cdot\right|$ stands for the standard euclidean norm in $\mathbb{R}^n$. 

Employing the change of variables formula, we can deduce that, for every fixed $r>0$, the functions $\phi{\left(r,\cdot\right)},\psi{\left(r,\cdot\right)}\colon\,B{\left(0_n,r\right)}\to\mathbb{R}$ are both Lebesgue integrable. Indeed, we have  $$\int\limits_{B{\left(0_n,r\right)}}{\phi{\left(r,x\right)}\mathrm{d}x}=\frac{2\pi^{\frac{n}{2}}}{\Gamma{\left(\frac{n}{2}\right)}}\int\limits_{0}^{r}{f{\left(r-t\right)}t^{n-1}\mathrm{d}t}=\frac{2\pi^{\frac{n}{2}}\left(n-1\right)!}{\Gamma{\left(\frac{n}{2}\right)}} A_{n,f,0}{\left(r\right)},$$ where for the first equality we employed the polar coordinates change of variables formula for the radial functions (see, e.g., \cite[Theorem 26.20 in page 695]{yeh2014real}). Similarly follows the result for the other function, $\psi$, for which we also note that, in view of \hyperref[prqst]{Proposition \ref*{prqst}}, we have $$\int\limits_{B{\left(0_n,r\right)}}{\psi{\left(r,x\right)}\mathrm{d}x}\neq 0,\text{ }\forall r>0.$$ 

We now claim that if $\frac{f}{g}\colon\,\left[0,\infty\right)\to\left[-\infty,\infty\right]$ is Lebesgue-almost everywhere (strictly) monotonic, then the well defined function $\frac{\int\limits_{B{\left(0_n,\cdot\right)}}{\phi{\left(\cdot,x\right)}\mathrm{d}x}}{\int\limits_{B{\left(0_n,\cdot\right)}}{\psi{\left(\cdot,x\right)}\mathrm{d}x}}\colon\,\left(0,\infty\right)\to\mathbb{R}$ is (strictly) monotonic of the same (strict) monotonicity. Indeed, from \hyperref[r00]{Theorem \ref*{r00}} we have that $\frac{A_{n,f,0}}{A_{n,g,0}}\colon\,\left(0,\infty\right)\to\mathbb{R}^n$ is (strictly) monotonic of the same (strict) monotonicity as of $\frac{f}{g}$ and the result then follows since $$\frac{A_{n,f,0}}{A_{n,g,0}}=\frac{\int\limits_{B{\left(0_n,\cdot\right)}}{\phi{\left(\cdot,x\right)}\mathrm{d}x}}{\int\limits_{B{\left(0_n,\cdot\right)}}{\psi{\left(\cdot,x\right)}\mathrm{d}x}}.$$ 
\end{enumerate}

\begin{appendices}

\section{The L'H\^opital rule for monotonicity via differential calculus}
\label{the lhopital rule for monotonicity via differential calculus}

We need the following straightforward extension to unbounded intervals of a well known result (see, e.g., \cite[Theorem 8.3.3 in page 229]{choudary2014real}), the proof of which is omitted. 

\begin{theorem}[the Rolle theorem]
\label{rll}
Consider 
\begin{enumerate}
\item a compact interval $I\subseteq\left[-\infty,\infty\right]$ and
\item a function $f\colon I\to\mathbb{R}$, such that 
\begin{enumerate}[label=\roman*.]
\item $f$ is continuous and
\item $\left.f\right|_{I^\circ}$ is differentiable. 
\end{enumerate}
\end{enumerate}
If $f{\left(\partial I\right)}$ is a singleton, then there exists a point $\xi\in I^\circ$, such that $f'{\left(\xi\right)}=0$. 
\end{theorem}

We also need the following extension of point 2. of \hyperref[prqst]{Proposition \ref*{prqst}}. 

\begin{proposition}
\label{prqstd}
Consider
\begin{enumerate}
\item a natural number $n\in\mathbb{N}$, 
\item an interval $I\subseteq\left[-\infty,\infty\right]$ when $n=1$ or $I\subseteq\mathbb{R}$ when $n\neq 1$, 
\item a point $c\in I$ and 
\item a function $f\colon\,I\to\mathbb{R}$, such that 
\begin{enumerate}[label=\roman*.]
\item $\left.f\right|_{I\cap\mathbb{R}}$ is $n$-times differentiable and
\item $f^{\left(n\right)}{\left(x\right)}\neq 0$, for all $x\in I\cap\mathbb{R}$. 
\end{enumerate}
\end{enumerate}
Then $${\left({R_{n-1,f,c}}^{\left(k\right)}\right)}^{-1}{\left(\left\{0\right\}\right)}=\left\{c\right\},\text{ }\forall k\in\left\{0,\dots,n-1\right\}.$$ 
\end{proposition}
\begin{proof}
To begin with, we have the equalities $${R_{n-1,f,c}}^{\left(k\right)}{\left(c\right)}=0,\text{ }\forall k\in\left\{0,\dots,n-1\right\}.$$ 

We assume that there exists a natural number $k\in \left\{0,\dots,n-1\right\}$ and a point $x\in I\setminus\left\{c\right\}$, such that ${R_{n-1,f,c}}^{\left(k\right)}{\left(x\right)}=0$. In view of the above sequence of equalities, we inductively apply \hyperref[rll]{Theorem \ref*{rll}} $n-k$ times, in order to deduce that there exists a point $\xi\in\left(\min{\left\{c,x\right\}},\max{\left\{c,x\right\}}\right)$, such that $f^{\left(n\right)}{\left(\xi\right)}=0$, which contradicts the assumption of the non vanishing $f^{\left(n\right)}$. 
\end{proof}

With \hyperref[prqstd]{Proposition \ref*{prqstd}} at hand, \hyperref[r2]{Theorem \ref*{r2}} is properly stated in the context of differential calculus. We now proceed to its proof. 

\begin{proof}[Proof of Theorem 2.] 
Since $g'{\left(x\right)}\neq 0$, for all $x\in I\cap\mathbb{R}$, from the Darboux theorem we have that $g'$ preserves a non zero sign, hence $g$ is strictly monotonous. Hence inverse function of $g$, $g^{-1}\colon\,g{\left(I\right)}\to I$ is well defined. Additionally, $g^{-1}$ is differentiable with $${\left(g^{-1}\right)}'=\frac{1}{g'\circ g^{-1}}.$$ 

Arguing as in the proof of \hyperref[r00]{Theorem \ref*{r00}}, it suffices to show the result for $g$ being strictly increasing and $f$ being (strictly) increasing. Therefore, we make such assumptions. From the strict monotonicity of $g$, the function $g^{-1}$ is also strictly increasing. 

Now, we consider the function $h\coloneqq f\circ g^{-1}\colon\,g{\left(I\right)}\to\mathbb{R}$, which is differentiable, due to the chain rule, with $$h'=f'\circ g^{-1}{\left(g^{-1}\right)}'=\frac{f'}{g'}\circ g^{-1},$$ thus $h'$ is (strictly) increasing as a composition of a strictly increasing function and a (strictly) increasing function. Hence $h$ is (strictly) convex. 

We then consider two arbitrary $x_1,x_2\,\in I\setminus\left\{c\right\}$, such that $x_1<x_2$. Since $g{\left(x_1\right)}<g{\left(x_2\right)}$, from the Galvani lemma we deduce that $$\frac{h{\left(g{\left(x_1\right)}\right)}-h{\left(g{\left(c\right)}\right)}}{g{\left(x_1\right)}-g{\left(c\right)}}\leq\frac{h{\left(g{\left(x_2\right)}\right)}-h{\left(g{\left(c\right)}\right)}}{g{\left(x_2\right)}-g{\left(c\right)}}\text{ }\left(\frac{h{\left(g{\left(x_1\right)}\right)}-h{\left(g{\left(c\right)}\right)}}{g{\left(x_1\right)}-g{\left(c\right)}}<\frac{h{\left(g{\left(x_2\right)}\right)}-h{\left(g{\left(c\right)}\right)}}{g{\left(x_2\right)}-g{\left(c\right)}}\right),$$ or else $$\frac{f{\left(x_1\right)}-f{\left(c\right)}}{g{\left(x_1\right)}-g{\left(c\right)}}\leq\frac{f{\left(x_2\right)}-f{\left(c\right)}}{g{\left(x_2\right)}-g{\left(c\right)}}\text{ }\left(\frac{f{\left(x_1\right)}-f{\left(c\right)}}{g{\left(x_1\right)}-g{\left(c\right)}}<\frac{f{\left(x_2\right)}-f{\left(c\right)}}{g{\left(x_2\right)}-g{\left(c\right)}}\right).$$
\end{proof}

\begin{proof}[Proof of Theorem 4.] 
It is only left to show the result for $n>1$ (with $I\subseteq\mathbb{R}$), thus we make such an assumption. 

To begin with, in view of \hyperref[prqstd]{Proposition \ref*{prqstd}} we have the following sequence of equalities $$\frac{{R_{n-1,f,c}}^{\left(k\right)}}{{R_{n-1,g,c}}^{\left(k\right)}}=\frac{{R_{n-1,f,c}}^{\left(k\right)}-{R_{n-1,f,c}}^{\left(k\right)}{\left(c\right)}}{{R_{n-1,g,c}}^{\left(k\right)}-{R_{n-1,g,c}}^{\left(k\right)}{\left(c\right)}},\text{ }\forall x\in I,\text{ }\forall k\in\left\{0,1,\dots,n-1\right\}.$$ 

Additionally, the following $$\frac{{R_{n-1,f,c}}^{\left(n\right)}}{{R_{n-1,g,c}}^{\left(n\right)}}=\frac{f^{\left(n\right)}}{g^{\left(n\right)}}$$ is true.  

Now, we inductively apply \hyperref[r1]{Theorem \ref*{r1}} $n$ times, in order to get that both $$\left.\frac{R_{n-1,f,c}}{R_{n-1,g,c}}\right|_{I\cap\left[-\infty,c\right)}\text{ and }\left.\frac{R_{n-1,f,c}}{R_{n-1,g,c}}\right|_{I\cap\left(c,\infty\right]}$$ are (strict) monotonic of the same (strict) monotonicity as of $\frac{f^{\left(n\right)}}{g^{\left(n\right)}}$. If $c\in \partial I$, then the proof is complete. 

Next, we deal with the case where $c\in I^\circ$. From the above (strict) monotonicity we deduce that the one sided limits to $c$ of these functions exist in $\left[-\infty,\infty\right]$, i.e. $$\lim\limits_{x\to c^-}{\frac{R_{n-1,f,c}{\left(x\right)}}{R_{n-1,g,c}{\left(x\right)}}}\in\left[-\infty,\infty\right]\ni\lim\limits_{x\to c^+}{\frac{R_{n-1,f,c}{\left(x\right)}}{R_{n-1,g,c}{\left(x\right)}}}.$$ Moreover, by the use of \textit{the Lagrange form of the remainder} (see, e.g., \cite[Theorem 8.4.1 in page 235]{choudary2014real}) we have that $$R_{n-1,f,c}{\left(x\right)}=\frac{f^{\left(n\right)}{\left(\xi_{f,x}\right)}}{n!}{\left(x-c\right)}^n,\text{ for some }\xi_{f,x}\in\left(\min{\left\{x,c\right\}},\max{\left\{x,c\right\}}\right),\text{ }\forall x\in I$$ and $$R_{n-1,g,c}{\left(x\right)}=\frac{g^{\left(n\right)}{\left(\xi_{g,x}\right)}}{n!}{\left(x-c\right)}^n,\text{ for some }\xi_{g,x}\in\left(\min{\left\{x,c\right\}},\max{\left\{x,c\right\}}\right),\text{ }\forall x\in I.$$ Therefore, $$\lim\limits_{x\to c^-}{\frac{R_{n-1,f,c}{\left(x\right)}}{R_{n-1,g,c}{\left(x\right)}}}=\lim\limits_{x\to c^-}{\frac{f^{\left(n\right)}{\left(\xi_{f,x}\right)}}{g^{\left(n\right)}{\left(\xi_{g,x}\right)}}}\in\mathbb{R}\ni\lim\limits_{x\to c^+}{\frac{f^{\left(n\right)}{\left(\xi_{f,x}\right)}}{g^{\left(n\right)}{\left(\xi_{g,x}\right)}}}=\lim\limits_{x\to c^+}{\frac{R_{n-1,f,c}{\left(x\right)}}{R_{n-1,g,c}{\left(x\right)}}},$$ since the function $\frac{f^{\left(n\right)}}{g^{\left(n\right)}}$ is (strictly) monotonous. By the same reason we deduce that $$\lim\limits_{x\to c^-}{\frac{R_{n-1,f,c}{\left(x\right)}}{R_{n-1,g,c}{\left(x\right)}}}\text{ }\left\{
\begin{aligned}
\leq\text{ }\left(<\right)\\
\geq\text{ }\left(>\right)
\end{aligned}\right\}\text{ }\lim\limits_{x\to c^+}{\frac{R_{n-1,f,c}{\left(x\right)}}{R_{n-1,g,c}{\left(x\right)}}}\text{ if }\frac{f^{\left(n\right)}}{g^{\left(n\right)}}\text{ is }\left\{
\begin{aligned}
\text{(stricty) increasing}\\
\text{(stricty) decreasing}
\end{aligned}\right\},$$ thus $$\frac{R_{n-1,f,c}{\left(x_1\right)}}{R_{n-1,g,c}{\left(x_1\right)}}\text{ }\left\{
\begin{aligned}
\leq\text{ }\left(<\right)\\
\geq\text{ }\left(>\right)
\end{aligned}\right\}\text{ }\frac{R_{n-1,f,c}{\left(x_2\right)}}{R_{n-1,g,c}{\left(x_2\right)}}\text{ if }\frac{f^{\left(n\right)}}{g^{\left(n\right)}}\text{ is }\left\{
\begin{aligned}
\text{(stricty) increasing}\\
\text{(stricty) decreasing}
\end{aligned}\right\},$$ for every $x_1,x_2\,\in I$, such that $x_1<c<x_2$, which completes the proof. 
\end{proof} 
\end{appendices}


\bibliographystyle{plain}
\bibliography{mybibfile}\label{bibliography}
\addcontentsline{toc}{chapter}{Bibliography}

\end{document}